\documentclass{amsart}

\usepackage{amssymb,amsmath,amsfonts, amsthm}
\usepackage{amscd,enumerate,graphicx}
\newtheorem{theorem}{Theorem}[section]

\newtheorem{corollary}[theorem]{Corollary}
\theoremstyle{remark}
\newtheorem{remark}[theorem]{Remark}

\theoremstyle{definition}

\numberwithin{equation}{section}

\title{Genera of two-bridge knots and epimorphisms of their knot groups}

\author{Masaaki Suzuki}
\address{Department of Frontier Media Science, 
Meiji University, 
4-21-1 Nakano, Nakano-ku, Tokyo, 164-8525, Japan}
\email{macky@fms.meiji.ac.jp}

\author{Anh T. Tran}
\address{Department of Mathematical Sciences, 
The University of Texas at Dallas, Richardson,
TX 75080, USA}
\email{att140830@utdallas.edu}



\begin{document}

\maketitle

\begin{abstract}
Let $K,K'$ be two-bridge knots of genus $n,k$ respectively. 
We show the necessary and sufficient condition of $n$ in terms of $k$ 
that there exists an epimorphism from the knot group of $K$ onto that of $K'$. 
\end{abstract}

\section{Introduction}\label{intro}

Let $K$ be a knot in $S^3$ and $G(K)$ the knot group, that is, 
the fundamental group of the complement of $K$ in $S^3$.  
We denote by $g(K)$ the genus of $K$. 
Recently, many papers have investigated epimorphisms between knot groups. 
In particular, Simon's conjecture in \cite{kirbylist}, 
which states that every knot group maps onto at most finitely many knot groups, 
was settled affirmatively in \cite{agolliu}. 
In the same Kirby's problem list \cite{kirbylist}, Simon also proposed another conjecture. 
Namely, if there exists an epimorphism from $G(K)$ onto $G(K')$, 
then is $g(K)$ greater than or equal to $g(K')$? 
This problem is also mentioned in \cite{kitano-suzuki2}. 
It is known that if there exists an epimorphism from $G(K)$ onto $G(K')$, 
then the Alexander polynomial of $K$ is divisible by that of $K'$. 
Moreover, Crowell \cite{Crowell} showed that 
the genus of an alternating knot is equal to a half of the degree of the Alexander polynomial. 
Then the above conjecture is true for alternating knots, especially two-bridge knots. 

In this paper, 
we give a more explicit condition on genera of two-bridge knots $K$ and $K'$
such that there exists an epimorphism between their knot groups. 
As a corollary, we show that 
if there exists an epimorphism from $G(K)$ onto $G(K')$, then $g(K) \geq 3 g(K') - 1$. 

A knot is called {\it minimal} if its knot group admits epimorphisms onto the knot groups of only the trivial knot and itself. 
Many types of minimal knots are already shown in 
\cite{kitano-suzuki3}, \cite{szk}, \cite{Burde}, \cite{MPV}, \cite{N}, \cite{NT}, and \cite{NST}.  
By using the main theorem of this paper, we obtain several types of minimal knots. 
%As mentioned in Section \ref{sect:ors}, 
%a two-bridge knot corresponds to a rational number. 
%Any rational number has continued fraction expansions. 
For example, 
a two-bridge knot of genus $2$ is minimal 
if and only if it is not the two-bridge knot $C[2a,4b,4a,2b]$ in Conway's notation 
for any non-zero integers $a,b$. 
%let $K$ be a two-bridge knot of genus $2$. 
%If a rational number corresponding to $K$ can not be expressed as 
%a continued fraction expansion 
%$[2a, 4b, 4a, 2b]$, 
%then $K$ is minimal. 

\section{Ohtsuki-Riley-Sakuma Construction}\label{sect:ors}

In this section, we review some known facts about two-bridge knots, 
see \cite{bzh} and \cite{murasugi} for example. 
Especially, we recall Ohtsuki-Riley-Sakuma construction of epimorphisms between two-bridge knot groups. 

It is known that a two-bridge knot corresponds to a rational number %$\frac{q}{p}$
and that it can be expressed as a continued fraction
\[
 %\frac{q}{p} = 
 [a_1,a_2, \ldots, a_{m-1},a_m] = 
\frac{1}{a_1 + \frac{1}{a_2 + \frac{1}{\ddots \frac{1}{a_{m-1} + \frac{1}{a_m}}}}} ,
\]
where $a_1 > 0$. 
We define the {\it length} of the continued fraction  to be 
\[
\ell ([a_1,a_2, \ldots, a_{m-1},a_m]) = m .
\] 
Note that the length depends on the choice of continued fractions. 
For example, we can delete zeros in a continued fraction  by using 
\begin{eqnarray*}
 [a_1,a_2,\ldots,a_{i-2}, a_{i-1},0,a_{i+1},a_{i+2}, \ldots, a_m] \\
= [a_1,a_2,\ldots,a_{i-2}, a_{i-1} + a_{i+1}, a_{i+2},\ldots, a_m] .
\end{eqnarray*}
Then, we can reduce the length by $2$, 
if the continued fraction contains a $0$. 

\begin{theorem}[Ohtsuki-Riley-Sakuma \cite{ORS}, Agol \cite{agol}, Aimi-Lee-Sakuma \cite{ALS}]\label{thm:ors}
Let $K(r), K(\tilde{r})$ be $2$-bridge knots, 
where $r = [a_1,a_2,\ldots,a_m]$. %is a rational number and $a_i > 0$. 
There exists an epimorphism $\varphi : G(K(\tilde{r})) \to G(K(r))$ 
%then $\varphi$ is meridional and 
if and only if $\tilde{r}$ can be written as 
\[
 \tilde{r} = 
[\varepsilon_1 {\bf a}, 2 c_1, 
\varepsilon_2 {\bf a}^{-1}, 2 c_2, 
\varepsilon_3 {\bf a}, 2 c_3, 
\varepsilon_4 {\bf a}^{-1}, 2 c_4, 
\ldots, 
\varepsilon_{2n} {\bf a}^{-1}, 2 c_{2n}, \varepsilon_{2n+1} {\bf a}] , 
\]
where ${\bf a} = (a_1, a_2,\ldots,a_m), {\bf a}^{-1} = (a_m, a_{m-1},\ldots,a_1)$, 
$\varepsilon_i = \pm 1 \, \, (\varepsilon_1 = 1)$, and $c_i \in {\mathbb Z}$. 
\end{theorem}

Remark that 
we can exclude the case where $c_i=0$ and $\varepsilon_i \cdot \varepsilon_{i+1}=-1$ 
without loss of generality (see \cite{szk} for details). 

A continued fraction $[a_1,a_2, \ldots,a_m]$ is called {\it even} 
if all $a_i$'s are even integers. 
Moreover, it is called {\it reduced} if all $a_i$'s are non-zero. 

\section{Main Theorem}\label{sect:mainthm}

First, we define a set $S_k$ as follows: 
\[
S_k = {\mathbb N} \cap \left(\bigcup_{r = 1}^{k-2} [(2 r + 1) k + r + 1, (2 r + 3) k - r - 2] \right) .
\]
For $j \in {\mathbb Z}$, we let ${\mathbb Z}_{\geq j}$ denote 
the set of all integers greater than or equal to $j$. 

In this section, we show the following theorem. 

\begin{theorem}\label{thm:mainthm}
Let $K'$ be a two-bridge knot of genus $k$. 
There exists a two-bridge knot $K$ of genus $n$ such that 
the knot group $G(K)$ admits an epimorphism onto $G(K')$ 
if and only if 
\[
n \in {\mathbb Z}_{\geq (3k-1)} \setminus S_k .
\]
\end{theorem}

\begin{proof}%[Proof of Theorem \ref{thm:mainthm}]
Recall that the length of the reduced even continued fraction corresponding to a two-bridge knot is 
twice the genus of the knot, see \cite{Cromwell} for example. 
A continued fraction of a rational number corresponding to $K'$ can be written as 
$[a_1,a_2, \ldots, a_{2k}]$ where all $a_i$'s are even and non-zero, 
%$a_i \neq 0$, 
since the genus of $K'$ is $k$. 
Suppose that there exists an epimorphism from $G(K)$ onto $G(K')$. 
By Theorem \ref{thm:ors}, a rational number corresponding to $K$ admits a continued fraction 
in the form 
\[
[\varepsilon_1 {\bf a}, 2 c_1, 
\varepsilon_2 {\bf a}^{-1}, 2 c_2, 
\varepsilon_3 {\bf a}, 2 c_3, 
\varepsilon_4 {\bf a}^{-1}, 2 c_4, 
\ldots, 
\varepsilon_{2r} {\bf a}^{-1}, 2 c_{2r}, \varepsilon_{2r+1} {\bf a}] 
\]
where ${\bf a} = (a_1,a_2, \ldots, a_{2k})$. 
As mentioned in Section \ref{sect:ors}, 
if $c_i = 0$, then we can reduce the length of  the continued fraction by $2$. 
After deleting $0$, the length of the continued fraction of $K$ is 
\[
2n  = (2 r + 1) \ell ([a_1,a_2, \ldots, a_{2k}]) + \sum_{i = 1}^{2r} w_i =  2 (2 r + 1) k + \sum_{i = 1}^{2r} w_i
\]
where 
\[
w_i = 
\left\{
\begin{array}{cl}
1 & \mbox{ if } c_i \neq 0 \\
-1 & \mbox{ if } c_i = 0 \\
\end{array}
\right. .
\]
We define $\ell$ as 
\[
\ell = \frac{1}{2} \sum_{i = 1}^{2r} w_i = r - \sharp \{ i \, | \, c_i = 0, 1 \leq i \leq 2r \} .
\]
Then $-r \leq \ell \leq r$ and $n = (2 r + 1) k + \ell$. Namely, 
\[
n \in  {\mathbb N} \cap \left(\bigcup_{r \in {\mathbb N}} [(2 r + 1) k - r, (2 r + 1) k +r] \right) .
\]
Here if $r \geq k - 1$, 
each interval does not have a gap with the next interval. 
%$(2 r + 1) k + r \geq (2 (r +1) + 1) k - (r +1)$ and then
Therefore the complement of the set to which $n$ belongs is 
\begin{equation}
 {\mathbb N} \cap \left([1,3 k -2] \cup \bigcup_{r = 1}^{k-2} [(2 r + 1) k + r + 1, (2 r + 3) k - r -2] \right)  \label{eq:sk}.
\end{equation}
Conversely, if $n$ belongs to ${\mathbb Z}_{\geq (3k-1)} \setminus S_k$, we can construct 
a two-bridge knot $K$ of genus $n$ whose knot group admits an epimorphism onto $G(K')$ as above. 
\end{proof}

\begin{corollary}\label{cor:genusinequality}
Let $K$ be a two-bridge knot and $K'$ a knot.  
If there exists an epimorphism $\varphi : G(K) \to G(K')$, 
then 
\begin{equation}\label{eq:genusinequality}
g(K) \geq 3 g(K') -1 . 
\end{equation}
\end{corollary}

\begin{remark}
We denote by $c(K)$ the crossing number of a knot $K$. 
Let $K$ be a two-bridge knot and suppose that 
there exists an epimorphism from $G(K)$ onto the knot group $G(K')$ of another knot $K'$. 
By the previous paper \cite{szk}, the following inequality holds 
\begin{equation}\label{eq:crossinginequality}
c(K) \geq 3 c(K') .
\end{equation}
Moreover, for a given two-bridge knot $K'$, 
we can construct a two-bridge knot $K$ with any crossing number satisfying the inequality (\ref{eq:crossinginequality}) such that $G(K)$ admits an epimorphim onto $G(K')$. 
However, Theorem \ref{thm:mainthm} implies 
we can not always construct a two-bridge knot $K$ of any genus 
even if it satisfies the inequality (\ref{eq:genusinequality}). 
More precisely, 
if $g(K) \geq 3 g(K') -1$ but $g(K) \in S_{g(K')}$, 
then there does not exist an epimorphism from $G(K)$ onto $G(K')$.  
Note that the cardinality of $S_k$ is 
\[
\sum_{r=1}^{k-2} ((2r+ 3)k - r - 2) - ((2r+1)k + r + 1) +1 = (k-1)(k-2) 
\]
and that of the set (\ref{eq:sk}) is 
\[
3k - 2 + \# S_k = 3k - 2 + (k-1)(k-2) = k^2 .
\]
\end{remark}

\section{Small genus}

In this section, we see some examples of small genus. 
Namely, 
for a small given $n$, we describe the continued fractions of two-bridge knots $K$ of genus $n$ 
which admit epimorphisms onto another knot $K'$ of genus $k$.  
Note that by the argument of the proof of Theorem \ref{thm:mainthm}, we have the following 
\begin{align}
&(2 r + 1) k - r \leq n \leq (2 r + 1) k + r  \label{ineq:knr}, \\
& \sharp \{ i \, | \, c_i = 0, 1 \leq i \leq 2r \} \label{eq:ci} = (2 r + 1) k + r - n .
\end{align}

{\it Case: $n=2$.} 
By Corollary \ref{cor:genusinequality}, the genus of $K'$ is $1$. 
Then we can take $[2a, 2b]$ for a continued fraction of $K'$, where $a,b \in {\mathbb Z} \setminus \{0 \}$. 
The inequality (\ref{ineq:knr}) implies $r=1$ and the equation (\ref{eq:ci}) implies that all $c_i$'s are $0$. 
Therefore the continued fraction of $K$ is 
\[
[2a, 2b, 0, 2b, 2a, 0, 2a,2b] = [2a, 4b, 4a, 2b] . 
\]
Furthermore, if a continued fracion of a two bridge knot of genus $2$ can not be expressed in this form, 
then this knot is minimal. 

{\it Case: $n=3$.} 
Similarly, the genus of $K'$ is $1$ and $[2a, 2b]$ can be taken as a continued fraction of $K'$. 
The inequality (\ref{ineq:knr}) implies $r=1$ or $2$.  
When $r=1$, one $c_i$ is $0$ and the other $c_i$ is not $0$ by the equation (\ref{eq:ci}). 
Then the continued fraction of $K$ is 
\[
[2a, 2b, 0, 2b, 2a, 2 c_2, 2 \varepsilon_3 a,2 \varepsilon_3 b] = [2a, 4b, 2a, 2c_2, 2 \varepsilon_3 a, 2 \varepsilon_3 b]  
\]
up to mirror image, where $\varepsilon_3 = \pm 1$. 
When $r=2$, all $c_i$'s are $0$ by the equation (\ref{eq:ci}). 
Then the continued fraction of $K$ is 
\[
[2a, 2b, 0, 2b, 2a, 0, 2a, 2b, 0, 2b, 2a, 0, 2a, 2b] = [2a, 4b, 4a, 4b, 4a, 2b]. 
\]

{\it Case: $n=4$.} 
The genus of $K'$ is $1$ and a continued fraction of $K'$ is $[2a, 2b]$. 
The inequality (\ref{ineq:knr}) implies $r=1,2,3$.  
When $r=1$, all $c_i$'s are not $0$. 
Then the continued fraction of $K$ is 
\[
[2a, 2b, 2 c_1, 2\varepsilon_2 b, 2 \varepsilon_2 a, 2 c_2, 2 \varepsilon_3 a,2 \varepsilon_3 b]
\]
up to mirror image, where $\varepsilon_2, \varepsilon_3 = \pm 1$. 
When $r=2$, three $c_i$'s are $0$ and one $c_i$ is not $0$. 
Then the continued fraction of $K$ is 
\[
[2a, 2b, 0, 2b, 2a, 0, 2a, 2b, 0, 2b, 2a, 2 c_4, 2 \varepsilon_5 a, 2 \varepsilon_5 b] =
[2a, 4b, 4a, 4b, 2a, 2 c_4, 2 \varepsilon_5 a, 2 \varepsilon_5 b]
\]
or
\begin{align*}
&[2a, 2b, 0, 2b, 2a, 0, 2a, 2b, 2 c_3, 2 \varepsilon_4 b, 2 \varepsilon_4 a, 0, 2 \varepsilon_4 a, 2 \varepsilon_4 b] \\
&= [2a, 4b, 4a, 2b, 2 c_3, 2 \varepsilon_4 b, 4 \varepsilon_4 a, 2 \varepsilon_4 b]
\end{align*}
up to mirror image, where $\varepsilon_4, \varepsilon_5 = \pm 1$. 
When $r=3$, all $c_i$'s are $0$. 
Then the continued fraction of $K$ is 
\begin{align*}
&[2a, 2b, 0, 2b, 2a, 0, 2a, 2b, 0, 2b, 2a, 0, 2a, 2b, 0, 2b, 2a, 0, 2a, 2b] \\
&= [2a, 4b, 4a, 4b, 4a, 4b,4a,2b]. 
\end{align*}

{\it Case: $n=5$.} 
In this case, the genus of $K'$ is $1$ or $2$ by Corollary \ref{cor:genusinequality}.
First, we consider the case that the genus of $K'$ is $1$ and that a continued fraction of $K'$ is $[2a, 2b]$. 
We list the continued fractions of $K$: 
\begin{align*}
&[2a, 2b, 0, 2b, 2a, 0, 2a, 2b, 2 c_3, 2 \varepsilon_4 b, 2 \varepsilon_4 a, 2 c_4, 2 \varepsilon_5 a, 2 \varepsilon_ 5b] \\
&= [2a, 4b, 4a, 2b, 2 c_3, 2 \varepsilon_4 b, 2 \varepsilon_4 a, 2 c_4, 2 \varepsilon_5 a, 2 \varepsilon_ 5b], \\
&[2a, 2b, 0, 2b, 2a, 2 c_2, 2 \varepsilon_3 a, 2 \varepsilon_3 b, 0, 2 \varepsilon_3 b, 2 \varepsilon_3 a, 2 c_4, 2 \varepsilon_5 a, 2 \varepsilon_5 b] \\
&= [2a, 4b, 2a, 2 c_2, 2 \varepsilon_3 a, 4 \varepsilon_3 b, 2 \varepsilon_3 a, 2 c_4, 2 \varepsilon_5 a, 2 \varepsilon_5 b], \\
&[2a, 2b, 2 c_1, 2 \varepsilon_2 b, 2 \varepsilon_2 a, 0, 2 \varepsilon_2 a, 2 \varepsilon_2 b, 0, 2 \varepsilon_2 b, 2 \varepsilon_2 a, 2 c_4, 2 \varepsilon_5 a, 2 \varepsilon_5 b] \\
&= [2a, 2b, 2 c_1, 2 \varepsilon_2 b, 4 \varepsilon_2 a, 4 \varepsilon_2 b, 2 \varepsilon_2 a, 2 c_4, 2 \varepsilon_5 a, 2 \varepsilon_5 b], \\
&[2a, 2b, 0, 2b, 2a, 2 c_2, 2 \varepsilon_3 a, 2 \varepsilon_3 b, 2 c_3, 2 \varepsilon_4 b, 2 \varepsilon_4 a, 0, 2 \varepsilon_4 a, 2 \varepsilon_4 b] \\
&= [2a, 4b, 2a, 2 c_2, 2 \varepsilon_3 a, 2 \varepsilon_3 b, 2 c_3, 2 \varepsilon_4 b, 4 \varepsilon_4 a, 2 \varepsilon_4 b], \\
&[2a, 2b, 0, 2b, 2a, 0, 2a, 2b, 0, 2b, 2a, 0, 2a, 2b, 0, 2b, 2a, 2 c_6, 2 \varepsilon_7 a, 2 \varepsilon_7 b] \\
&= [2a, 4b, 4a, 4b, 4a, 4b, 2a, 2 c_6, 2 \varepsilon_7 a, 2 \varepsilon_7 b], \\
&[2a, 2b, 0, 2b, 2a, 0, 2a, 2b, 0, 2b, 2a, 0, 2a, 2b, 2 c_5, 2 \varepsilon_6 b, 2 \varepsilon_6 a, 0, 2 \varepsilon_6 a, 2 \varepsilon_6 b] \\
&= [2a, 4b, 4a, 4b, 4a, 2b, 2 c_5, 2 \varepsilon_6 b, 4 \varepsilon_6 a, 2 \varepsilon_6 b], \\
& [2a, 2b, 0, 2b, 2a, 0, 2a, 2b, 0, 2b, 2a, 2 c_4, 2 \varepsilon_5 a, 2 \varepsilon_5 b, 0, 2 \varepsilon_5 b, 2 \varepsilon_5 a, 0, 2 \varepsilon_5 a, 2 \varepsilon_5 b], \\
&= [2a, 4b, 4a, 4b, 2a, 2 c_4, 2 \varepsilon_5 a, 4 \varepsilon_5 b, 4 \varepsilon_5 a, 2 \varepsilon_5 b], \\
&[2a, 2b, 0, 2b, 2a, 0, 2a, 2b, 0, 2b, 2a, 0, 2a, 2b, 0, 2b, 2a, 0, 2a, 2b, 0, 2b, 2a, 0, 2a, 2b] \\
&= [2a, 4b, 4a, 4b, 4a, 4b, 4a, 4b, 4a, 2b]
\end{align*}
up to mirror image, where $\varepsilon_2, \varepsilon_3, \ldots,\varepsilon_7  = \pm 1$. 
Next, the genus of $K'$ is $2$ and the continued fraction of $K'$ is $[2a,2b,2c,2d]$, 
where $a,b,c,d \in {\mathbb Z} \setminus \{0 \}$. 
Then the continued fraction of $K$ is 
\[
[2a, 2b, 2c, 2d, 0, 2d, 2c, 2b, 2a, 0, 2a, 2b, 2c, 2d] 
= [2a, 2b, 2c, 4d, 2c, 2b, 4a, 2b, 2c, 2d]. 
\]

By using the above arguments, 
we obtain a criterion whether a given two-bridge knot of genus up to $5$ is minimal. 

\begin{theorem}
A two-bridge knot $K$ of genus up to $5$ is not minimal if and only if 
a continued fraction of a rational number corresponding to $K$ can be expressed as one of the following: 
\begin{align*}
& [2a, 4b, 4a, 2b], \\
& [2a, 4b, 2a, 2c_2, 2 \varepsilon_3 a, 2 \varepsilon_3 b], [2a, 4b, 4a, 4b, 4a, 2b], \\
& [2a, 2b, 2 c_1, 2\varepsilon_2 b, 2 \varepsilon_2 a, 2 c_2, 2 \varepsilon_3 a,2 \varepsilon_3 b],
[2a, 4b, 4a, 4b, 2a, 2 c_4, 2 \varepsilon_5 a, 2 \varepsilon_5 b],\\
& [2a, 4b, 4a, 2b, 2 c_3, 2 \varepsilon_4 b, 4 \varepsilon_4 a, 2 \varepsilon_4 b],  
[2a, 4b, 4a, 4b, 4a, 4b,4a,2b], \\ 
&[2a, 4b, 4a, 2b, 2 c_3, 2 \varepsilon_4 b, 2 \varepsilon_4 a, 2 c_4, 2 \varepsilon_5 a, 2 \varepsilon_ 5b], \\
&[2a, 4b, 2a, 2 c_2, 2 \varepsilon_3 a, 4 \varepsilon_3 b, 2 \varepsilon_3 a, 2 c_4, 2 \varepsilon_5 a, 2 \varepsilon_5 b], \\
&[2a, 2b, 2 c_1, 2 \varepsilon_2 b, 4 \varepsilon_2 a, 4 \varepsilon_2 b, 2 \varepsilon_2 a, 2 c_4, 2 \varepsilon_5 a, 2 \varepsilon_5 b], \\
&[2a, 4b, 2a, 2 c_2, 2 \varepsilon_3 a, 2 \varepsilon_3 b, 2 c_3, 2 \varepsilon_4 b, 4 \varepsilon_4 a, 2 \varepsilon_4 b], \\
&[2a, 4b, 4a, 4b, 4a, 4b, 2a, 2 c_6, 2 \varepsilon_7 a, 2 \varepsilon_7 b], \\
&[2a, 4b, 4a, 4b, 4a, 2b, 2 c_5, 2 \varepsilon_6 b, 4 \varepsilon_6 a, 2 \varepsilon_6 b], \\
&[2a, 4b, 4a, 4b, 2a, 2 c_4, 2 \varepsilon_5 a, 4 \varepsilon_5 b, 4 \varepsilon_5 a, 2 \varepsilon_5 b], \\
&[2a, 4b, 4a, 4b, 4a, 4b, 4a, 4b, 4a, 2b], \\
& [2a, 2b, 2c, 4d, 2c, 2b, 4a, 2b, 2c, 2d] 
\end{align*}
where $a,b,c,d,c_i \neq 0$ and $\varepsilon_i = \pm 1$.
\end{theorem}

\section*{Acknowledgements}
The first author was partially supported by KAKENHI 
(No.\ 16K05159), Japan Society for the Promotion of Science, Japan. 
The second author was partially supported by a grant from the Simons Foundation 
(No.\ 354595 to AT).


\begin{thebibliography}{99}

\bibitem{agol}I.\ Agol, 
{\em The classification of non-free 2-parabolic generator Kleinian groups}, 
Slides of talks given at Austin AMS Meeting and Budapest Bolyai conference, July 2002,
Budapest, Hungary.

\bibitem{agolliu}I.\ Agol and Y.\ Liu, 
{\em Presentation length and Simon's conjecture}, 
J. Amer. Math. Soc. {\bf 25} (2012), 151--187.

\bibitem{ALS}S.\ Aimi, D.\ Lee, and M.\ Sakuma, 
{\em Parabolic generating pairs of $2$-bridge link groups}, 
in preparation. 

%\bibitem{BBRW}M.\ Boileau, S.\ Boyer, A.\ Reid, and S.\ Wang, 
%{\em Simon's conjecture for two-bridge knots}, 
%Comm. Anal. Geom. {\bf 18} (2010), 121--143.

%\bibitem{BB}M.\ Boileau and S.\ Boyer, 
%{\em On character varieties, sets of discrete characters, and nonzero degree maps},
%Amer. J. Math. {\bf 134} (2012), 285--347.

\bibitem{Burde} G. Burde, {\em $SU(2)$-representation spaces for two-bridge knot groups}, Math. Ann. \textbf{288} (1990), 103--119.

\bibitem{bzh}G.\ Burde, H.\ Zieschang, and M.\ Heusener, 
Knots, 
De Gruyter Studies in Mathematics, {\bf 5}, 2014.

\bibitem{Cromwell}P.\ Cromwell, 
{\em Knots and Links}, 
Cambridge University Press, Cambridge, 2004. xviii+328 pp.

\bibitem{Crowell}R.\ Crowell, 
{\em Genus of alternating link types}, 
Ann. of Math. {\bf 69} (1959), 258–275.

%\bibitem{CL}J.\ Cha and C.\ Livingston, 
%KnotInfo: Table of Knot Invariants, 
%http://www.indiana.edu/\~{}knotinfo, March 16, 2016. 
%
%\bibitem{chasuzuki}J.\ Cha and M.\ Suzuki, 
%{\em Non-meridional epimorphisms of knot groups}, 
%Algebr. Geom. Topol. {\bf 16} (2016), 1135--1155.

%\bibitem{FW} C.D. Feustel, W. Whitten, 
%{\em Groups and complements of knots}, 
%Canad. J. Math. {\bf 30} (1978) 1284--1295.

%\bibitem{HKMS}K.\ Horie, T.\ Kitano, M.\ Matsumoto, and M.\ Suzuki, 
%{\em A partial order on the set of prime knots with up to $11$ crossings}, 
%J. Knot Theory Ramifications {\bf 20} (2011), 275--303. 

\bibitem{kirbylist}R.\ Kirby, 
{\em Problems in low-dimensional topology, Geometric topology (Athens, GA, 1993) (Rob Kirby,
ed.)}, 
AMS/IP Stud. Adv. Math., vol. 2, Amer. Math. Soc., Providence, RI, 1997, 35--473.

%\bibitem{kitano-suzuki1}T.\ Kitano and M.\ Suzuki, 
%{\em A partial order on the knot table},
%Experimental Math. {\bf 14} (2005), 385--390.

\bibitem{kitano-suzuki2}T.\ Kitano and M.\ Suzuki, 
{\em Twisted Alexander polynomial and a partial order on the set of prime knots}, 
Geom. Topol. Monogr. {\bf 13} (2008), 307-321. 

\bibitem{kitano-suzuki3}T.\ Kitano and M.\ Suzuki, 
{\em Some minimal elements for a partial order of prime knots}, 
preprint. 

\bibitem{MPV}M.\ Macasieb, K.L.\ Petersen and R.\ van Lujik, 
{\em On character varieties of two-bridge knot groups}, 
Proc. London Math. Soc., {\bf 103} (2011), 473--504.

\bibitem{murasugi}K.\ Murasugi, 
Knot theory and its applications, 
Birkhauser (1996).

\bibitem{N}F.\ Nagasato,
{\em On minimal elements for a partial order of prime knots}, 
Topology Appl. {\bf 159} (2012), 1059--1063.

\bibitem{NST}F.\ Nagasato, M.\ Suzuki, and A.\ Tran, 
{\em On minimality of two-bridge knots}, 
Internat. J. Math. {\bf 28} (2017), 11 pages.

\bibitem{NT}F.\ Nagasato and A.\ Tran, 
{\em Some families of minimal elements for a partial ordering on prime knots}, 
Osaka J. Math. {\bf 53} (2016), 1029--1045. 

\bibitem{ORS}T.\ Ohtsuki, R.\ Riley and M.\ Sakuma,
{\em Epimorphisms between 2-bridge link groups},
Geom. Topol. Monogr. {\bf 14} (2008), 417--450.

%\bibitem{riley}R.\ Riley, 
%{\em Nonabelian representations of 2-bridge knot groups}, 
%Quart. J. Math. Oxford Ser. (2) {\bf 35} (1984), 191--208.

%\bibitem{silverwhitten}D.\ Silver and W.\ Whitten, 
%{\em Knot group epimorphisms II}, preprint.

\bibitem{szk}M.\ Suzuki, 
{\em Epimorphisms between two bridge knot groups 
and their crossing numbers}, 
to appear in Algebr. Geom. Topol. 

%\bibitem{Tran} A.\ Tran, 
%{\em Reidemeister torsion and Dehn surgery on twist knots}, 
%to appear in Tokyo Journal of Mathematics.

\end{thebibliography}
\end{document}